\documentclass[11pt,a4paper,leqno]{amsart}
\usepackage[utf8]{inputenc}
\nonstopmode \textwidth=16.00cm
\usepackage{amssymb,color,tikz, cancel}
\usepackage{amsmath,amscd,amsfonts,amsthm,verbatim}
\usepackage{tikz-cd}
\usepackage{enumitem}
\usepackage[all]{xy}
\usetikzlibrary{matrix} 

\usepackage{xcolor}
\usepackage{blindtext}
\usepackage{multicol}

\numberwithin{equation}{section} 

\definecolor{ffzzqq}{rgb}{1,0.6,0}
\definecolor{qqqqff}{rgb}{0,0,1}
\definecolor{ffqqqq}{rgb}{1,0,0}
\definecolor{wwzzqq}{rgb}{0.4,0.6,0}
\definecolor{zzwwff}{rgb}{0.6,0.4,1}
\usepackage{pgfplots}
\pgfplotsset{compat=1.15}
\usepackage{mathrsfs}
\usetikzlibrary{arrows}

\newcommand{\ZZ}{\mathbb{Z}}

\DeclareMathOperator{\Hom}{Hom}

\DeclareMathOperator{\Ann}{Ann}
\DeclareMathOperator{\reg}{reg}

 \def\cocoa{{\hbox{\rm C\kern-.13em
 o\kern-.07em C\kern-.13em o\kern-.15em A}}}
\newtheorem{theorem}{Theorem}[section]
\newtheorem{lemma}[theorem]{Lemma}
\newtheorem{proposition}[theorem]{Proposition}
\newtheorem{Problem}[theorem]{Problem}

 \theoremstyle{definition}
\newtheorem{definition}[theorem]{Definition} \theoremstyle{remark}
\newtheorem{remark}[theorem]{Remark}
\newtheorem{example}[theorem]{Example}

\definecolor{MyDarkGreen}{cmyk}{0.7,0,1,0}

\begin{document}

\title[Complete intersection algebras with  binomial Macaulay dual generator]{Complete intersection algebras with \\ binomial Macaulay dual generator}

\author[R.Di Gennaro]{Roberta Di Gennaro}
\address{Di Gennaro: Dipartimento di Matematica e Applicazioni “Renato Caccioppoli” \\
 Universit\`a degli Studi di Napoli Federico II\\
  80126 Napoli, Italy}
\email{digennar@unina.it}

\author[R.M. Miró-Roig]{Rosa M. Miró-Roig}
\address{Mir\'o-Roig: Department de Mathem\`atiques i Inform\`atica\\
  Universitat de Barcelona\\
  08007 Barcelona, Spain}
\email{miro@ub.edu, ORCID 0000-0003-1375-6547}

\thanks{The first author  was partially supported by GNSAGA- INdAM. The second author was supported by the grant PID2020-113674GB-I00 and by the program \lq\lq Professori Visitatori of GNSAGA- INdAM\rq\rq}

\subjclass[2020]{ Primary 14M10; Secondary 13E10, 13C40.}

 \keywords{Weak Lefschetz property, Strong Lefschetz property, complete intersection, Macaulay duality}

\begin{abstract} In this paper, we characterize all Artinian complete intersection $K$-algebras $A_F$ whose Macaulay dual generator $F$ is a binomial. In addition, we prove that  such 
 complete intersection Artinian $K$-algebras $A_F$ satisfy the Strong Lefschetz property.
\end{abstract}

\date{January 2025}

\maketitle

\section{Introduction}

Artinian Gorenstein standard graded $K$-algebras  (also known as Poincaré duality $K$-algebras)  have long been ubiquitous in many fields of mathematics, including algebraic geometry, commutative algebra, algebraic topology, combinatorics, etc.  Central to the study of Artinian Gorenstein  graded $K$-algebras is the concept of Macaulay dual generator. Indeed, 
Macaulay-Matlis  duality establishes a reversing one-to-one correspondence between  Artinian Gorenstein $K$-algebras $R/I$ and a single polynomial, their inverse system $I^{\perp}=<F>$ also called the Macaulay dual generator.  This duality not only  provides a parametrization of Artinian Gorenstein $K$-algebras $R/I$ but also translates algebraic properties of $R/I$ into questions about the homogeneous polynomial $F$, streamlining their exploration.  In this context and in a natural way we are led to pose the following question: To characterize the  homogeneous polynomials $F$ such that their Macaulay's dual $I=\Ann_R(F)$ is an Artinian complete intersection ideal.

So far, only a limited number of instances of this problem have been reported. It is widely known
that codimension 2 Artinian Gorenstein graded $K$-algebras and complete intersections coincide. We also know that if  $F$ is a monomial then $\Ann_R(F)$ is an Artinian  monomial complete intersection ideal. In \cite{HWW}
the case of quadratic complete intersections is considered while in 
 \cite{ADF+24}, the authors  initiated a systematic study for key properties of Artinian Gorenstein
$K$-algebras  $A_F=$$R/\Ann_R(F)$ having as a Macaulay dual generator  a binomial $F$. They were able to demonstrate that in codimension 3
 all such algebras satisfy the strong Lefschetz property (see Definition \ref{WLP+SLP}) and they established an explicit characterization of when they are complete intersection.

In this short note we are going to
investigate to what extent similar result holds for Artinian Gorenstein $K$-algebras of arbitrary codimension and Macaulay dual generator a binomial: Can we characterize {\em all} Artinian  complete intersection  ideal $I=\Ann_R(F)$
 having as a Macaulay dual generator a binomial $F$? Do they have SLP?

 Over the last several decades, this last question has been deeply studied in a much more general set-up giving rise to the problem  of  whether or not Artinian graded $K$-algebras have the WLP/SLP. If they do, it has strong implications for the Hilbert function;
 for example,  their Hilbert function is unimodal.
One of the most important kinds of $K$-algebras that have been intensively studied from the point of
view of WLP/SLP is the class of Artinian Gorenstein graded $K$-algebras and, in particular,  Artinian complete intersections graded $K$-algebras. In this paper, we give new contributions to the following problem:
Do {\em all}  Artinian complete intersections $K$-algebras have the WLP/SLP, in characteristic 0?

This is currently known only in codimension two  or three, although partial results in higher codiemnsion are also known. 
Let us quickly summarize what is known so far. 
A {\em general} complete intersection with fixed generator degrees
has the WLP/SLP and  {\em all} complete intersection in one and two
variables have the WLP. In \cite{HMNW}, Harima, Migliore, Nagel and Watanabe showed that any Artinian codimension  three complete intersection $K$-algebra  has the WLP. On the other hand, Stanley \cite{S} and Watanabe \cite{W} showed that \emph{any} monomial complete intersection, regardless of the number of variables, has the WLP/SLP. In spite of the great activity in this area the question whether {\em any} Artinian complete intersection $K$-algebra  satisfies the WLP/SLP  remains open. In this paper, we extent Watanabe and Stanley result for Artinian complete intersections with Macaulay dual generator a monomial to Artinian complete intersection with Macaulay dual generator a binomial with the hope that our approach will provide new insight in this long standing open problem of characterizing the Macaulay dual generators of Artinian complete intersection $K$-algebras.

\vskip 2mm
Our paper is structured as follows: in Section 2,  we fix notation and we gather together the basic facts on Artinian Gorenstein $K$-algebras and Macaulay-Matlis duality needed in the sequel. Section 3 is the heart of the paper since it contains our main result, namely, the characterization of {\em all} binomials $F$ such that $\Ann_R(F)$ is a complete intersection ideal (see Theorem \ref{thm1}). In Section 4, we prove that all Artinian complete intersection $K$-algebras with Macaulay dual generator a binomial satisfy SLP (see Theorem \ref{thm2}). We end the paper with a short section where we collect some questions which naturally arise from our work.

\medskip \noindent  {\bf Acknowledgement.} 
The second author thanks the  Dipartimento di Matematica e Applicazioni “Renato Caccioppoli” of the 
 Universit\`a degli Studi di Napoli Federico II for the warm hospitality, where most of the research was performed.

\section{Notation and background}
 This section contains the basic definitions and results on Artinian Gorenstein $K$-algebras
 and it lays the groundwork for the results in the later sections.

Throughout this paper, $K$ will be an algebraically closed field of characteristic zero. 
Given a standard graded Artinian $K$-algebra $A=R/I$, where $R=K[x_1,\dots,x_n]$ and $I$ is a homogeneous ideal of $R$,
we denote the Hilbert function of $A$ by $HF_A\colon \mathbb{Z} \longrightarrow \mathbb{Z}$, with $HF_A(j)=\dim _K[A]_j=\dim _K[R/I]_j$. 
Since $A$ is Artinian, its Hilbert function is
encoded in its \emph{$h$-vector} $h=(h_0,\dots ,h_d)$, where $h_i=HF_A(i)>0$ and $d$ is the largest index with this property. The integer $d$ is called the \emph{socle degree of} $A$  or the {\em regularity} of $A$ and denoted $\reg(A)$. We refer to the number $n=\dim(R)$, which is equal to the height or codimension of $I$, also as the {\em codimension} of $A$.

 \subsection{Gorenstein algebras and Macaulay duality}\label{s: Macaulay duality}

A graded Artinian $K$-algebra $A$  with socle degree $d$ is said to be  {\em Gorenstein} if its
socle $Soc(A):=(0:m_A)$ is a one dimensional $K$-vector space whose elements have degree $d$.

Let $R=K[x_1,\ldots,x_n]$ be a polynomial ring and let $S=K[X_1,\ldots,X_n]$ be a divided power algebra (see \cite[Appendix A]{IK} or \cite[Appendix A.2.4]{Ei}), regarded as a $R$-module with the \emph{contraction} action
\[
x_i\circ X_j^k=
\begin{cases}X_j^{k-1}\delta_{ij} \  & \text{if} \ k>0\\ 0 & \text{otherwise,}\\
\end{cases}
\]
where $\delta_{ij}$ is the Kronecker delta. We regard $R$ as a graded $K$-algebra with $\deg X_i=\deg x_i$.

For each degree $i\geq 0$, the action of $R$ on $S$ defines a non-degenerate $K$-bilinear pairing
\begin{equation}
\label{eq:MDPairing}
    R_i \times S_i \longrightarrow K \text{ with } (f,F) \longmapsto f \circ F.
\end{equation}
Thus for each $i\geq 0$ we have an isomorphism of $K$-vector spaces $S_i\cong \Hom_K(R_i,K)$ given by $F\mapsto\left\{f\mapsto f\circ F\right\}$. The above perfect pairing also yields an isomorphism $A_i^*\cong A_{d-i}$. So, the h-vector of such $K$-algebras is symmetric.

 It is well known by work of F.\,S.\,Macaulay \cite{Mac} that an Artinian $K$-algebra $A=R/I$ is Gorenstein with socle degree $d$ if and only if  $$I=\Ann_R(F)=\{f\in R\mid f\circ F=0\}$$ for some homogeneous polynomial $F\in S_d$.  The polynomial $F$ is called  {\em Macaulay dual generator} for $A$; it is unique up to a scalar multiple. To emphasize the relationship between $F$ and the Artinian Gorenstein ring defined by its annihilator, we write $A_F=R/\Ann_R(F)$. 

By the non-degenerate $K$-bilinear pairing \eqref{eq:MDPairing}, we have:
$$
HF_A(t)=\dim_K(A_{d-t})=\dim_K \left\langle x_0^{i_0}\cdots x_N^{i_N} \circ F  : i_0+\cdots+i_N=t\right\rangle.
$$

We end this preliminary version  recalling a result  that will be strongly used in the next section.

 \begin{lemma}\label{lem: socle fits}
     Any surjective degree-preserving homomorphism $\pi:A\to B$ between Artinian Gorenstein $K$-algebras with the same socle degree is an isomorphism.
 \end{lemma}
 \begin{proof}
Set $d=\reg(A)=\reg(B)$. Let $f\in \ker(\pi)$ and assume that $f\neq 0$. Since $A$ is an Artinian Gorenstein $K$-algebra, there exists $g\in A$ such that $fg\in Soc(A)$ and $fg\neq 0$.  By hypothesis $\pi$ is a degree-preserving homomorphism between two Artinian Gorenstein $K$-algebras $A$ and $B$ with socle degree $d$. Therefore,  $\pi$ induces an isomorphism $\pi:A_d\xrightarrow{\cong} B_d$ and, since $0\neq fg\in A_d$, we conclude  $0\neq \pi(fg)=\pi(f)\pi(g)$. This is a contradiction, because $\pi(f)=0$. Thus, $\ker(\pi)=0$ and hence $\pi$ is an isomorphism.
 \end{proof}


\section{Complete intersections with binomial dual generator}

In this section, we address the following longstanding open problem in commutative algebra (see \cite{ADF+24}, \cite{E}, \cite{HWW} and its references list):

\begin{Problem} \label{pblm}
    To characterize the homogeneous forms $F\in K[X_1,\cdots ,X_n]$  of degree $d$ whose annihilator $\Ann _R(F)$ is a complete intersection ideal.
\end{Problem}

If $F\in K[X_1,\cdots ,X_n]$ is a monomial, the answer is well known. In fact, for $F=X_1^{a_1}\cdots X_n^{a_n}$ we have $\Ann_R(F)=(x_1^{a_1+1},\cdots ,x_n^{a_n+1})$ and $A_F=R/\Ann_R(F)$ satisfies SLP (see \cite{S} and \cite{W}).
We will now tackle  the above problem for the first open case, namely, binomials.
Since  $K$ is algebraically closed, it contains all roots of unity and hence for any binomial $G\in S$ there is a binomial $F=m_1-m_2\in S$ so that $m_1,m_2$ are monomials and, moreover, $A_G\cong A_F$ (see \cite[Lemma 3.1]{ADF+24}). So, it is natural to try to answer the above question (Problem \ref{pblm}) for the first open case, namely, when $F$ is a binomial and characterize Artinian complete intersection $K$-algebras having as a Macaulay dual generator a binomial.

From now on, we fix a binomial $F=m_1-m_2\in K[X_1,\dots ,X_n]$ with $m_1, m_2$ monomials. After reordering the variables, if necessary, any binomial can be factored as 
\[
F=X_1^{a_1}\cdots X_n^{a_n}(X_1^{b_1}\cdots X_r^{b_r}-X_{r+1}^{b_{r+1}}\cdots X_n^{b_n}), \quad\text{ where } 1\le r\le n-1 \text{ and}
\]
\begin{align}\label{eq: binomial}
m_1 & =  X_1^{a_1+b_1}\cdots X_r^{a_r+b_r} X_{r+1}^{a_{r+1}}\cdots X_n^{a_n}, \nonumber \\
 m_2 & = X_1^{a_1}\cdots X_r^{a_r} X_{r+1}^{a_{r+1}+b_{r+1}}\cdots X_n^{a_n+b_n}, \nonumber \\
 g &=\gcd(m_1,m_2)  =   X_1^{a_1}\cdots X_n^{a_n} \nonumber \\
 d &= \sum _{i=1}^na_{i} + \sum_{i=1}^rb_{i} =  \sum _{i=1}^na_{i} + \sum_{i=r+1}^n b_{i}.
 \end{align}

 \begin{remark}
     First of all we observe that we can assume that $b_i\ne 0$ for any $0\le i \le n$. If, for instance, $b_1=0$ (analogous argument if $b_i=0$ for any other index $i>1$) we write $F=F_1F_2$ where $F_1=X_1^{a_1}$ and   $
F=X_2^{a_2}\cdots X_n^{a_n}(X_2^{b_2}\cdots X_{n-1}^{b_{n-1}}-X_n^{b_n}).
$ It holds:
$$A_F=A_{F_1}\otimes A_{F_2}=K[x_1]/(x_1^{a_1+1})\otimes _K K[x_2,\cdots ,x_n]/\Ann_R(F_2).$$ Therefore, $A_F$ is a complete intersection if and only if $A_{F_2}$ is a complete intersection. Thus,  we have to consider only the cases with all $b_i\ne 0$.
 \end{remark}

Our goal will be to determine in terms of $a_i$ and $b_i$ when the Artinian Gorenstein algebra $A_F$ with Macaulay dual generator a binomial $F$ is a complete intersection. In other words, our next theorem is a new contribution towards a characterization of the inverse systems of complete intersections  (see Problem \ref{pblm}).

\begin{theorem}\label{thm1}
  Assume $n\ge 3$ and  consider the binomial 
$$
F=X_1^{a_1}X_2^{a_2}\cdots X_n^{a_n}(X_1^{b_1}X_2^{b_2}\cdots X_r^{b_r}-X_{r+1}^{b_{r+1}}\cdots X_n^{b_n}) \quad \text{ with }\quad \sum_{i=1}^rb_i=\sum_{j=r+1}^nb_j>0.
$$
 Then $A_F$ is a complete intersection if and only if, after changing coordinates, if necessary, we have $r=n-1$ and there exists an integer $1\le i\le n-1$ such that $a_i<qb_i$ being $q=\left \lfloor \frac{a_n+1}{b_n}\right\rfloor$.
 Moreover, in this case we define $m= \min \left\{\left  \lfloor\frac{a_j}{b_j} \right \rfloor \mid 1\leq j\leq n-1\right\}+1$ and we have $$\Ann_R(F)=(x_1^{a_1+b_1+1},x_2^{a_2+b_2+1},\cdots  ,x_{n-1}^{a_{n-1}+b_{n-1}+1}, G)$$
where $$G=x_n^{a_n+1}+\sum _{j=1}^m(x_1^{b_1}\cdots x_{n-1}^{b_{n-1}})^{j}x_n^{a_n+1-jb_n}.$$
\end{theorem}
\begin{proof} 

We assume that $$
F=X_1^{a_1}X_2^{a_2}\cdots X_n^{a_n}(X_1^{b_1}X_2^{b_2}\cdots X_{n-1}^{b_{n-1}}-X_n^{b_n}) \quad \text{ with }\quad \sum_{i=1}^{n-1}b_i=b_n>0.
$$ 
 with $0\le a_i<qb_i$ for some $1\le i \le n-1$ being $q=\left \lfloor \frac{a_n+1}{b_n}\right\rfloor$ and we will prove that $A_F$ is a complete intersection.
  We consider the ideal
 $$J=(x_1^{a_1+b_1+1},x_2^{a_2+b_2+1},\cdots  ,x_{n-1}^{a_{n-1}+b_{n-1}+1}, G)$$
where $$G=x_n^{a_n+1}+\sum _{j=1}^m(x_1^{b_1}\cdots x_{n-1}^{b_{n-1}})^{j}x_n^{a_n+1-jb_n}.$$ Notice that $a_n+1-mb_n\geq 0$ because $m\leq q$.
It holds:
\begin{enumerate}
    \item $J$ is a complete intersection Artinian ideal. In fact, it is generated by a regular sequence.
    \item $\reg(R/J)=\reg(A_F)=\deg (F)$. In fact, $\reg(R/J)=\sum_{j=1}^{n}a_j+b_n=\reg(A_F)=\deg (F)$.
    \item $J\subset \Ann _R(F)$. In fact, $x_j^{a_j+b_j+1}\circ F=0$ for $1\le j \le n-1$ and  a straightforward computation gives $G\circ F=0$.
\end{enumerate}
Therefore, there is a surjective degree-preserving homomorphism $$\pi:R/J\to A_F=R/\Ann_R (F)$$ and applying Lemma \ref{lem: socle fits}
we conclude that $A_F$ is a complete intersection.

\vskip 2mm

Let us prove the converse. We distinguish two cases:

\vskip 2mm
\noindent \underline{Case 1:} We assume $r<n-1$. Let us write $$F=X_1^{a_1}X_2^{a_2}\cdots X_n^{a_n}(X_1^{b_1}X_2^{b_2}\cdots X_r^{b_r}-X_{r+1}^{b_{r+1}}\cdots X_n^{b_n}) $$ with $\sum_{i=1}^{r}b_i=\sum _{j=r+1}^nb_j>0$ and $r<n-1$. We will demonstrate that $\Ann_R(F)$ cannot be a complete intersection. We first observe that $S=\{x_i^{a_i +1} x_j^{a_j +1}| 1\leq i\leq r,\  r+1\leq j\leq n\}$ is part of a minimal system of generators of $\Ann_R(F)$ and, moreover, $J=(S) \subsetneq \Ann_R(F)$ because $\Ann_R(F)$ is an Artinian ideal while $J=(S)$ is not Artinian. As usual we denote by $\mu(I)$ the minimal number of generators of an ideal $I$. Since $$\mu(\Ann_R(F))>|S|\ge n$$ we immediately conclude that $A_F$ is not a complete intersection $K$-algebra.

\vskip 2mm
\noindent \underline{Case 2:} We assume $r=n-1$ and  $a_i\ge qb_i$ for all $1\le i\le n-1$. 
We will again show that $\Ann_R(F)$ is not a complete intersection. 
Let us assume that  $\Ann_R(F)$ is  a complete intersection and we will get a contradiction. We write  $$F=X_1^{a_1}\cdots X_n^{a_n}({X_1}^{b_1}\cdots {X_{n-1}}^{b_{n-1}}-{X_n}^{b_n})$$
 where without lose of generality we  assume  that $0<a_1+b_1\leq a_{2}+b_2\leq\cdots \leq a_{n-1}+b_{n-1}$. 
 
 \vskip 2mm
 \noindent {\bf Claim:} The monomials 
 $ x_1^{a_1+b_1+1},x_2^{a_2+b_2+1},\cdots , x_{n-1}^{a_{n-1}+b_{n-1}+1}$ 
are part of a minimal system of generators of $\Ann_R(F)$.
 
 \vskip 2mm
 \noindent {\em Proof of the Claim.}  Clearly $x_1^{a_1+b_1+1}\in \Ann_R(F)$. Hence, if $x_1^{a_1+b_1+1}$ is not minimal, there will exist a homogeneous form $f_a\in \Ann_R(F)$ of degree $a\le a_1+b_1$ of the following type 
 $$f_a=x_1^a+\sum_{(i_1, \cdots ,i_n)\ne (a,0,\cdots ,0)\atop i_1+\cdots +i_n=a} \alpha _{i_1,...,i_n}x_1^{i_1}\cdots x_n^{i_n}.$$
Since $f_a\in \Ann_R(F)$, we have 
\begin{equation}\label{computation}
\begin{array}{rcl} 0= f_a\circ F & = & f_a\circ X_1^{a_1}\cdots X_n^{a_n}({X_1}^{b_1}\cdots {X_{n-1}}^{b_{n-1}}-{X_n}^{b_n})\\ \\
& = & x_1^a \circ X_1^{a_1}\cdots X_n^{a_n}({X_1}^{b_1}\cdots {X_{n-1}}^{b_{n-1}}-{X_n}^{b_n}) + \\
& & ( \sum  \alpha _{i_1,...,i_n}x_1^{i_1}\cdots x_n^{i_n}.)\circ X_1^{a_1}\cdots X_n^{a_n}({X_1}^{b_1}\cdots {X_{n-1}}^{b_{n-1}}-{X_n}^{b_n}) \\\\
 & = & X_1^{a_1+b_1-a}X_2^{a_2+b_2}\cdots X_{n-1}^{a_{n-1}+b_{n-1}}X_n^{a_n}-{X_1}^{a_1-a}X_2^{a_2}\cdots {X_{n-1}}^{a_{n-1}}X_n^{a_n+b_n} + \\
 & & \sum_{i_1,\cdots ,i_n}  \alpha _{i_1,...,i_n} X_1^{a_1+b_1-i_1}X_2^{a_2+b_2-i_2}\cdots X_{n-1}^{a_{n-1}+b_{n-1}-i_{n-1}}X_n^{a_n-i_n} -\\
 & & \sum_{i_1,\cdots ,i_n}  \alpha _{i_1,...,i_n} {X_1}^{a_1-i_1}X_2^{a_2-i_2}\cdots X_{n-1}^{a_{n-1}-i_{n-1}}X_n^{a_n+b_n-i_n}
\end{array}
\end{equation}
where we assume $X_i^{h_i}=0$ if $h_i<0$, for some $1\leq i\leq n$. An intricate computation shows that there is no an n-ple $(a,0,\cdots ,0)\ne (i_1,\cdots ,i_n)\in \ZZ^n_{\ge 0}$ with $\sum _{j=1}^ni_j=a$ such that $$ X_1^{a_1+b_1-a}X_2^{a_2+b_2}\cdots X_{n-1}^{a_{n-1}+b_{n-1}}X_n^{a_n}$$ coincide with  either $$ X_1^{a_1+b_1-i_1}X_2^{a_2+b_2-i_2}\cdots X_{n-1}^{a_{n-1}+b_{n-1}-i_{n-1}}X_n^{a_n-i_n} \quad \text{or} \quad  {X_1}^{a_1-i_1}X_2^{a_2-i_2}\cdots X_{n-1}^{a_{n-1}-i_{n-1}}X_n^{a_n+b_n-i_n}.$$ Therefore, the  first monomial in (\ref{computation}) cannot be delayed, i.e.
$f_a\circ F\ne 0$ and, hence, we conclude that $x_1^{a_1+b_1+1}$ is part of a minimal system of generators. Similar argument can be applied for any variable $X_i$ with $1\le i \le n-1$ and this finishes the proof of the Claim.

\vskip 2mm
Since we are assuming that  $\Ann_R(F)$ is an Artinian  complete intersection ideal  we, necessarily, have:
$$\Ann_R(F)=(x_1^{a_1+b_1+1}, \cdots, x_{n-1}^{a_{n-1}+b_{n-1}+1},P)$$ for some homogeneous form $P\in S$. Let us compute the degree of $P$, We have:
 $$\reg (A_F)=\sum _{i=1}^{n-1}(a_1+b_1+1)+\deg P-n= \deg F.$$
 Therefore, $\deg (P)=a_n+1.$ Being $\Ann_R(F)$ is an Artinian ideal we must have 
  $$P=x_n^{a_n+1}+\sum _{(i_1, \cdots ,i_n)\ne (0,0,\cdots ,a_n+1)\atop i_1+\cdots +i_n=a_n+1} \alpha _{i_1,...,i_n}x_1^{i_1}\cdots x_n^{i_n}.$$

We will check that such $P\notin \Ann_R(F)$, namely, $P\circ F\ne 0$. Indeed, we have:
\begin{equation}\label{computation2}
\begin{array}{rcl} P\circ F & = & P\circ X_1^{a_1}\cdots X_n^{a_n}({X_1}^{b_1}\cdots {X_{n-1}}^{b_{n-1}}-{X_n}^{b_n})\\ \\
& = & x_n^{a_n+1} \circ X_1^{a_1}\cdots X_n^{a_n}({X_1}^{b_1}\cdots {X_{n-1}}^{b_{n-1}}-{X_n}^{b_n}) + \\
& & ( \sum  \alpha _{i_1,...,i_n}x_1^{i_1}\cdots x_n^{i_n}.)\circ X_1^{a_1}\cdots X_n^{a_n}({X_1}^{b_1}\cdots {X_{n-1}}^{b_{n-1}}-{X_n}^{b_n}) \\\\
 & = & -{X_1}^{a_1}X_2^{a_2}\cdots {X_{n-1}}^{a_{n-1}}X_n^{b_n-1} + \\
 & & \sum_{i_1,\cdots ,i_n}  \alpha _{i_1,...,i_n} X_1^{a_1+b_1-i_1}X_2^{a_2+b_2-i_2}\cdots X_{n-1}^{a_{n-1}+b_{n-1}-i_{n-1}}X_n^{a_n-i_n} -\\
 & & \sum_{i_1,\cdots ,i_n}  \alpha _{i_1,...,i_n} {X_1}^{a_1-i_1}X_2^{a_2-i_2}\cdots X_{n-1}^{a_{n-1}-i_{n-1}}X_n^{a_n+b_n-i_n}.
\end{array}
\end{equation}

We analyze separately the case $q=0$ from the case $q>0$. If $q=0$ or, equivalently, $a_n+1<b_n$ we get from the equality (\ref{computation2}) that the summand $-{X_1}^{a_1}X_2^{a_2}\cdots {X_{n-1}}^{a_{n-1}}X_n^{b_n-1}$ only cancels when $(i_1,\cdots ,i_{n-1},i_n)=(b_1,\cdots ,b_{n-1},a_n+1-b_n)$ which is not allowed because $a_n+1-b_n<0$. Therefore, if $q=0$ we have $P\circ F\ne 0$  and as we wanted to prove $P\notin \Ann_R(F)$.

From now on, we assume that $q>0$. Let us call
$$A_{i_1,\cdots ,i_n}:=  X_1^{a_1+b_1-i_1}X_2^{a_2+b_2-i_2}\cdots X_{n-1}^{a_{n-1}+b_{n-1}-i_{n-1}}X_n^{a_n-i_n} $$
and
$$B_{i_1,\cdots ,i_n}:={X_1}^{a_1-i_1}X_2^{a_2-i_2}\cdots X_{n-1}^{a_{n-1}-i_{n-1}}X_n^{a_n+b_n-i_n}.$$
We observe that 
$$A_{i_1,\cdots ,i_n}\neq A_{j_1,\cdots ,j_n} \ \ \ \text{ and } \ \ \  B_{i_1,\cdots ,i_n}\neq B_{j_1,\cdots ,j_n}, \ \ \ \text{if } \ \ (i_1,\cdots ,i_n)\neq (j_1,\cdots ,j_n).$$
Moreover, for $\sigma_s:=(sb_1,\ldots, sb_{n-1}, a_n-sb_n+1)$ with $s\geq 1$ we have
$$\begin{array}{rcl}A_{\sigma_1}&=&{X_1}^{a_1}X_2^{a_2}\cdots {X_{n-1}}^{a_{n-1}}X_n^{b_n-1},\\
A_{\sigma_{s+1}}&=&B_{\sigma_s} \text{ for } 1\le s \le q-1,\\
A_{i_1,\cdots ,i_n} & \neq &  B_{j_1,\cdots ,j_n}  \text{ if } (i_1,\cdots ,i_n)\ne \sigma _s \text{ for any }s
.\end{array}$$
We  write $a_n+1=qb_n+r$ with $0\leq r<b_n$. We have and $\sigma_{q}=(qb_1, \ldots, qb_{n-1}, r)$ and $\sigma_{q+1}=((q+1)b_1, \ldots, (q+1)b_{n-1}, r-b_n)$. Since $b_n-r<0$,  the summand $B_{\sigma_q}$  in the equality \eqref{computation2} will never be canceled. Therefore,   $P\circ F\ne 0$ and $P\notin \Ann_R(F)$.
\end{proof}

\begin{remark}
    The case $n=3$ of the above theorem was already proved in \cite[Corollary 4.3]{ADF+24} with an alternative proof that strongly uses the structure theorem  of the Gorenstein ideals of codimension 3 given by Buchsbaum and Eisenbud in \cite{BE}. It uses the fact that any codimension 3 Gorenstein ideal is generated by the pfaffians of a skew symmetric matrix.
\end{remark}

We will now illustrate  the above result with  examples.
\begin{example} Set $R=K[x,y,z,t,u,v]$ and $S=K[X,Y,Z,T,U,V]$.  $S$  is a divided power
algebra  regarded as a $R$-module with the contraction action.

  \vskip 2mm
    (1)  We take $$F=X^3YZ^4TU^2V^6(X^2YZT^2U-V^7)\in S.$$ Using Macaulay2 \cite{M2} we compute the annhilator of $F$ and we get
$$\Ann_R(F)=(y^3,u^4,t^4,z^6,x^6,v^7+x^2yzt^2u).$$
Therefore, $A_F$ is an Artinian complete intersection $K$-algebra as it is predict in Theorem \ref{thm1}.
    \vskip 2mm

    (2) We take $$F=XYZTUV^2(XYZ-V^3)\in S.$$ Using Macaulay2 \cite{M2} we compute the annhilator of $F$ and we get
$$\Ann_R(F)=(y^3,u^2,t^2,z^3,x^3,v^6, zv^3+xyz^2,yv^3+xy^2z,xv^3+x^2yz).$$
Therefore, $A_F$ is not an Artinian complete intersection $K$-algebra as it is predict in Theorem \ref{thm1}.
     \vskip 2mm

    (3) Finally, we take $$F=XYZTUV(XYZ-TV^2)\in S.$$ Using again  Macaulay2 \cite{M2} we compute the annhilator of $F$ and we get
$$\Ann_R(F)=(y^3,u^2,t^3,z^3,x^3,v^4,z^2v^2,y^2v^2,x^2v^2,z^2t^2,y^2t^2,x^2t^2,$$ $$xyz^2+ztv^2,xy^2z+ytv^2,x^2yz+xtv^2).$$
Therefore, $A_F$ is an Artinian non complete intersection $K$-algebra as it is predict in Theorem \ref{thm1}.
\end{example}

\section{Weak and Strong Lefschetz property for complete intersections}

This section is entirely devoted to prove that any artinian standard complete intersection graded $K$-algebra $A_F$ with dual Macaulay generator a binomial $F$ has the strong Lefschetz property.
\begin{definition} \label{WLP+SLP}
Let $A=R/I$ be a graded Artinian $K$-algebra. We say that $A$ has the {\em weak Lefschetz property} (WLP, for short)
if there is a linear form $\ell \in [A]_1$ such that, for all
integers $i\ge0$, the multiplication map
\[
\times \ell: [A]_{i}  \longrightarrow  [A]_{i+1}
\]
has maximal rank.
 In this case, the linear form $\ell$ is called a weak Lefschetz
element of $A$. If for the general form $\ell \in [A]_1$ and for an integer $j$ the
map $\times \ell:[A]_{j-1}  \longrightarrow  [A]_{j}$ does not have maximal rank, we will say that the ideal $I$ fails the WLP in
degree $j$.

$A$ has the {\em strong Lefschetz property} (SLP, for short) if there is a linear form $\ell \in [A]_1$ such that, for all
integers $i\ge0$ and $k\ge 1$, the multiplication map
\[
\times \ell^k: [A]_{i}  \longrightarrow  [A]_{i+k}
\]
has maximal rank.  Such an element $\ell$ is called a strong Lefschetz element for $A$.
\end{definition}

These properties play a crucial role in the study of
Artinian $K$-algebras, influencing their algebraic structure and leading to geometric and
combinatorial applications. To determine whether an Artinian standard graded $K$-algebra $A$ has the WLP or SLP seems a simple problem of linear algebra, but it has proven to be extremely elusive and  much more work on this topic remains to be done, see \cite{JMR} and its reference list for more information. 
In particular, the following longstanding problem remains open:

\begin{Problem}
    Does {\em any} complete intersection Artinian graded $K$-algebra satisfy WLP/SLP?
\end{Problem}

In codimension 2 any graded Artinian $K$-algebra has the SLP property \cite{Br} and in codimension 3 any complete intersection graded algebra has WLP (see \cite{HMNW}). 
As we have just pointed out a major unresolved issue is to determine whether all complete intersection rings of characteristic zero satisfy the WLP or
SLP.
Via Macaulay-Matlis duality, any  Artinian Gorenstein graded $K$-algebra $A_F$ with  Macaulay dual generator a monomial is a quotient of the polynomial ring by a monomial complete intersection and Stanley \cite{S} and Watanabe \cite{W} have shown that $A_F$ the SLP.
While  very recently codimension-three complete intersection  Artinian $K$-algebras with
binomial Macaulay dual generators are known to satisfy both WLP
and SLP \cite[Theorem ?.?]{ADF+24}, determining broader conditions that guarantee these properties in higher
codimensions is open and we solve it in  our next theorem.

We start recalling a couple of crucial results. First, the fact that 
the SLP is preserved by tensor product. To be precise we have:

\begin{proposition}\label{tensor}
    If $A$ and $B$ are graded Artinian $K$-algebras with the SLP then the
 tensor product $A\otimes _K B$ also has the SLP.
\end{proposition}
\begin{proof}
    See \cite[Proposition 11]{HW}.
\end{proof}

\begin{proposition}\label{keySLP}
 Let $A$ be a graded Artinian  $K$-algebra with the
SLP and let $g\in A[z]$ be a homogeneous monic polynomial. Then $B=A[z]/(g)$ has
 the strong Lefschetz property.
\end{proposition}
\begin{proof}
    See \cite[Corollary 29(1)]{HW}.
\end{proof}

\begin{theorem} \label{thm2}
Let $A_F$ be an Artinian complete intersection $K$-algebra of  codimension $n$. Assume that the Macaulay dual generator is a binomial. Then, $A_F$ satisfies SLP.
\end{theorem}

\begin{proof} By Theorem \ref{thm1} we know that after a change of coordinates we have 
$$
F=X_1^{a_1}X_2^{a_2}\cdots X_n^{a_n}(X_1^{b_1}X_2^{b_2}\cdots X_{n-1}^{b_{n-1}}-X_n^{b_n}) \quad with \quad \sum_{i=1}^{n-1}b_i=b_n>0.
$$ 
 with $a_i<qb_i$ for some $1\le i \le n-1$ and $q=\left \lfloor \frac{a_n+1}{b_n}\right\rfloor$. We also have
 $$\Ann_R(F)=(x_1^{a_1+b_1+1},x_2^{a_2+b_2+1},\cdots  ,x_{n-1}^{a_{n-1}+b_{n-1}+1}, G)$$
where $$G=x_n^{a_n+1}-\sum _{j=1}^m(x_1^{b_1}\cdots x_{n-1}^{b_{n-1}})^{j}x_n^{a_n+1-jb_n}$$
and $m= \min \left\{ \frac{a_j}{b_j}|1\leq j\leq n-1\right\}+1$.
Since $K$ has characteristic zero, applying \cite{S} and \cite{W} we obtain that the Artinian monomial complete intersection $K$-algebra $A=K[x_1,\cdots ,x_{n-1}]/I$ where $I=(x_1^{a_1+b_1+1},x_2^{a_2+b_2+1},\cdots  ,x_{n-1}^{a_{n-1}+b_{n-1}+1})$ satisfies both WLP and SLP. We now take the homogeneous monic polynomial $G\in A$. Applying  Proposition \ref{keySLP} we conclude that $A_F=K[x_1,\cdots ,x_n]/\Ann _R(F)\cong A[x_n]/(G)$ has SLP.
\end{proof}

\begin{remark}
    In \cite{ADF+24} the authors prove that, in codimension 3, any Gorenstein $K$-algebra with Macaulay dual generator a binomial has the SLP. In \cite{ADM+}, for arbitrary codimension, they only describe families of binomials $F$ whose Gorenstein $K$-algebra $A_F$ satisfies WLP 
\end{remark}

\section{Final comments and open problems}

In this last section, we  formulate the open problems emerging along this work and we add some concluding remarks. Along the paper we have dealt with  binomial Macaulay dual generator $F$ and we analyze whether the associated Artinian Gorenstein $K$-algebra $A_F$ is a complete intersection. First we would like to point out that the   Macaulay dual generator $F$ of an Artinian complete intersection $K$-algebra $A_F$ can have an arbitrary number of monomials. Indeed, we have:

\begin{proposition}
    For any integer $s\ge 1$ there is an Artinian complete intersection $K$-algebra of codimension $c\ge 2$ whose Macaulay dual generator $F$ is the sum of $s$ monomials. Moreover, $A_F$ satisfies the SLP 
\end{proposition}

\begin{proof}
If $c=2$ we take $F=X_1^{s-1}+X_1^{s-2}X_2+\cdots +X_2^{s-1}$ and we get and Artinian Gorenstein $K$-algebra $A_F$ which is obviously complete intersection since in codimension 2 Gorenstein ideals and complete intersection ideals coincide. By \cite{Br}, $A_F$ has the SLP.

Assume $c\ge 3$. Take $F=F_1F_2$ where $F_1=X_1^{s-1}+X_1^{s-2}X_2+\cdots +X_2^{s-1}$ and $F_2=X_3^{a_3}\cdots X_c^{a_c}$. We have just seen that $A_{F_1}$ is a complete intersection $K$-algebra which satisfies SLP. By \cite{S} and \cite{W} we know that $A_{F_2}$ is an Artinian monomial complete intersection $K$-algebra which also satisfies SLP (Indeed, $A_{F_2}=K[x_3,\cdots ,x_c]/(x_3^{a_3+1},\dots .x_c^{a_c+1})$). Applying Proposition \ref{tensor} we conclude that $A_F=A_{F_1}\otimes _K A_{F_2}$ is an Artinian complete intersection $K$-algebra verifying SLP.
\end{proof}

It is also worthwhile to point out that for $n\ge 3$, $d\ge 2$ and $(n,d)\ne (3,3)$ the Artinian Gorenstein $K$-algebra $A_F$ with  Macaulay dual generator  a generic form $F\in S=K[X_1,\cdots ,X_n]$ of degree $d$ is never a complete intersection. Indeed, $A_F$ is a compressed Artinian Gorenstein $K$-algebra. Therefore, the Macaulay dual generators are far from being generic and we are let to pose the following problem:

\begin{Problem}
   \begin{enumerate}
      \item To characterize the forms $F\in S$ of degree $d$ such that $A_F$ is a complete intersection.
      \item Given a form $F\in S$ such that  $A_F$ is a complete intersection, prove that $A_F$ has SLP.
   \end{enumerate} 
\end{Problem}

So far, the answer to the last Problem is only known for monomials and binomials Macaulay dual generators.
and only partial results
can be found in the literature.

\vskip 2mm 
 In
this paper, we assume throughout that our work is over a field of characteristic zero. However, we want to point out that
 the characteristic of the ground field $K$ plays a very interesting role in such questions, in fact it is easy to find complete intersections that fail WLP in positive characteristic. For example (\cite{BK}) $K[x,y,z]/(x^d,y^d,z^d)$ fails WLP in characteristic $2$ if $d$ is even. More precisely, Brenner and Kaid \cite{BK} completely characterized the characteristics in which monomial complete intersections in three variables generated by monomials all having the same degree have WLP, analogously Kustin and
Vraciu \cite{KV} made this study for at least four variables and their classification was completed by Lundquist and Nicklasson \cite{LN} for at least five variables.

\end{document}